\documentclass[ECP]{ejpecp} % replace ECP by EJP if needed.  

\usepackage[T1]{fontenc}
\usepackage[utf8]{inputenc}

\usepackage{tikz}
\SHORTTITLE{Exponential convergence to equilibrium for the $d$-dimensional East model} 

\TITLE{Exponential convergence to equilibrium for the $d$-dimensional East model
\thanks{This work has been supported by the ERC Starting Grant 680275 MALIG.}} % \thanks is optional. Insert line breaks with \\

\AUTHORS{%
  Laure Mar{\^e}ch\'e\footnote{LPSM UMR 8001, Universit\'e Paris Diderot, Sorbonne Paris Cit\'e, CNRS, 75013 Paris, France.
    \EMAIL{mareche@lpsm.paris}}
    }
\AMSSUBJ{60K35} % Edit. Separate items with ;
\SUBMITTED{May 2, 2019} % Edit.
\ACCEPTED{August 19, 2019} % Edit.

\VOLUME{24}
\YEAR{2019}
\PAPERNUM{55}
\DOI{https://doi.org/10.1214/19-ECP261}

\ABSTRACT{Kinetically constrained models (KCMs) are interacting particle systems on $\mathds{Z}^d$ 
with a continuous-time constrained Glauber dynamics, which were introduced by physicists to model the 
liquid-glass transition. One of the most well-known KCMs is the one-dimensional East model. 
Its generalization to higher dimension, the $d$-dimensional East model, is much less understood. 
Prior to this paper, convergence to equilibrium in the $d$-dimensional East model was proven 
to be at least stretched exponential, by Chleboun, Faggionato and Martinelli in 2015. We show that 
the $d$-dimensional East model exhibits exponential convergence to equilibrium in all settings 
for which convergence is possible.}

\begin{document}

%%%%%%%%%%%%%%%%%%%%%%%%%%%%%%%%%%%%%%%%%%%%%%%%%%%%%%%%%%%%%%%%%%%
%%                                                               %%
%% No need for \maketitle.                                       %%
%%                                                               %%
%%%%%%%%%%%%%%%%%%%%%%%%%%%%%%%%%%%%%%%%%%%%%%%%%%%%%%%%%%%%%%%%%%%

%%%%%%%%%%%%%%%%%%%%%%%%%%%%%%%%%%%%%%%%%%%%%%%%%%%%%%%%%%%%%%%%%%%
%%                                                               %%
%% Please replace what follows by the body of your article       %%
%% (up to the bibliography):                                     %%
%%                                                               %%
%%%%%%%%%%%%%%%%%%%%%%%%%%%%%%%%%%%%%%%%%%%%%%%%%%%%%%%%%%%%%%%%%%%

\section{Introduction}

Kinetically constrained models (KCMs) are interacting particle systems on graphs, in which each 
vertex (or \emph{site}) of the graph has state (or \emph{spin}) 0 or 1. Each site tries at rate 1 to 
\emph{update} its spin, that is to replace it by 1 with probability $p$ and by 0 with probability $1-p$, 
but the update is accepted only if a certain \emph{constraint} is satisfied, the constraint being of the 
form ``there are enough sites with spin zero around this site''.

KCMs were introduced by physicists to model the liquid-glass transition, which is an important 
open problem in condensed matter physics (see \cite{Ritort_et_al,Garrahan_et_al}). 
In addition to their physical interest, they are also mathematically challenging because the presence 
of the constraints gives them a very different behavior from classical Glauber dynamics 
and renders most of the usual tools ineffective.

A key feature of KCMs is the existence of blocked spin configurations, which 
makes the large-time behavior of KCMs hard to study, especially their relaxation 
to equilibrium when starting out of equilibrium. Indeed, worst case analysis does not help and standard 
coercive inequalities of the log-Sobolev type also fail. Furthermore, the dynamics of KCMs is not attractive, 
so coupling arguments that have proven very useful for other types of 
Glauber dynamics are here inefficient. Because of these difficulties, convergence to equilibrium has been proven 
only in a few models and under particular conditions 
(see \cite{Cancrini_et_al2010,Blondel_et_al2013,Chleboun_et_al2015,Mountford_FA1f}).

There is only one model for which exponentially fast relaxation to equilibrium 
was proven under general conditions (apart from some models on trees that use the same proof): 
the East model, whose base graph is $\mathbb{Z}$ and in which 
an update is accepted when the site at the left has spin 0.
Introduced by physicists in \cite{Jackle_et_al1991}, the East model is the most well-understood 
KCM (see \cite{review_East} for a review).

A natural generalization of the East model to $\mathbb{Z}^d$, introduced in 
\cite{Berthier_et_al2005}, is to accept updates at a site $x$ when $x-e$ has spin 0 
for some $e$ in the canonical basis of $\mathbb{R}^d$. The higher dimension makes this $d$-dimensional East model 
much harder to study than the unidimensional one, and until now 
the relaxation to equilibrium was only proved to be at least stretched exponential (\cite{Chleboun_et_al2015}).

In this article, we prove that the relaxation to equilibrium in the $d$-dimensional East dynamics 
is exponentially fast as soon as the initial configuration is not blocked. This also 
allowed us to prove that the \emph{persistence function}, which is the probability that a given site 
has not yet been updated, decays exponentially with time.

Our results, which are the first to hold for a KCM in dimension greater than 1 and for any $p$, 
may help to understand further the out-of-equilibrium behavior of the $d$-dimensional East model. 
Indeed, such an exponential relaxation result was key to proving ``shape theorems'' 
in one-dimensional models in \cite{Blondel2013,Ganguly_et_al2015,Blondel_et_al2018}.

This paper is organized as follows: we begin by presenting the notations and stating our results in 
Section \ref{sec_notations}, then we prove the exponential relaxation to equilibrium in 
Section \ref{sec_preuve_thm}, and finally we show the exponential decay of the persistence 
function in Section \ref{sec_preuve_cor}.

\section{Notations and results}\label{sec_notations}

We fix $d \in \mathbb{N}^*$. For any $\Lambda \subset \mathbb{Z}^d$, the $d$-dimensional East model 
(in the following, we will just call it ``East model'') in $\Lambda$ is a dynamics on $\{0,1\}^\Lambda$. 
The elements of $\Lambda$ will be called sites and 
the elements of $\{0,1\}^\Lambda$ will be called configurations. 
For any $\eta \in \{0,1\}^\Lambda$, $x \in \Lambda$, the value of $\eta$ at $x$ will be called the 
spin of $\eta$ at $x$ and denoted by $\eta(x)$.

If $f : \{0,1\}^{\Lambda} \mapsto \mathbb{R}$ is a function and $\Lambda' \subset \Lambda$, 
we say the support of $f$ is contained in $\Lambda'$ and we write $\mathrm{supp}(f) \subset \Lambda'$ 
when for any $\eta,\eta' \in \{0,1\}^\Lambda$ coinciding in $\Lambda'$, $f(\eta) = f(\eta')$. 
Moreover, the $\ell^\infty$-norm of $f$, denoted by $\|f\|_\infty$, is 
$\sup_{\eta \in \{0,1\}^{\Lambda}}|f(\eta)|$.

We denote $\{e_1,\dots,e_d\}$ the canonical basis of $\mathbb{R}^d$. 
For any $r \in \mathbb{R}^+$, we denote $\Lambda(r) = 
(\prod_{i=1}^d\{0,\dots,\lfloor r \rfloor \}) \setminus \{(0,\dots,0)\}$.

For any set $A$, $|A|$ will denote the cardinal of $A$.
For $\rho, \rho' \in \mathbb{R}$, we will use the abbreviation 
$\rho \wedge \rho' = \min(\rho, \rho')$.

\medskip

To define the East dynamics in $\Lambda \subset \mathbb{Z}^d$, we begin by fixing $p \in ]0,1[$. 
Informally, the East dynamics can be seen as follows: each site $x$, independently of all others, waits 
for a random time with exponential law of mean 1, then tries to update its spin, that is to replace it by 1 with
probability $p$ and by 0 with probability $1-p$, but the update is accepted if and only if one of the $x-e_i$ is at zero. 
Then $x$ waits for another random time with exponential law, etc. 

More rigorously, independently for each $x \in \Lambda$, we consider a sequence 
$(B_{x,n})_{n \in \mathbb{N}^*}$ of independent random 
variables with Bernoulli law of parameter $p$, and a sequence of times $(t_{x,n})_{n \in \mathbb{N}^*}$ 
such that, denoting $t_{x,0}=0$, the $(t_{x,n}-t_{x,n-1})_{n \in \mathbb{N}^*}$ 
are independent random  variables with exponential law of parameter 1, 
independent from $(B_{x,n})_{n \in \mathbb{N}^*}$. 
The dynamics is continuous-time, denoted by $(\eta_t)_{t \in \mathbb{R}^+}$, and evolves as follows.
For each $x \in \Lambda$, $n \in \mathbb{N}^*$, 
if there exists $i \in \{1,\dots,d\}$ such that $\eta_{t_{x,n}}(x-e_i) = 0$, 
then the spin at $x$ is replaced by $B_{x,n}$ at time $t_{x,n}$. We then say there was an update 
at $x$ at time $t_{x,n}$, or that $x$ was updated at time $t_{x,n}$. 
(If there are sites $x-e_i$, $x \in \Lambda$, $i \in \{1,\dots,d\}$ that are not in $\Lambda$, we need 
to fix the state of their spins in order to run the dynamics.) One can use the arguments 
in Section 4.3 of \cite{Swart2017} to see that this dynamics is well-defined.

For any $\eta \in \{0,1\}^{\Lambda}$, we denote the law of the 
dynamics starting from the configuration $\eta$ by 
$\mathbb{P}_\eta$, and the associated expectation by $\mathbb{E}_\eta$. If the initial 
configuration follows a law $\nu$ on $\{0,1\}^{\Lambda}$, the law and expectation of the 
dynamics will be respectively denoted by $\mathbb{P}_\nu$ and $\mathbb{E}_\nu$.
In the remainder of this work, we will always consider the dynamics on $\mathbb{Z}^d$ unless stated otherwise. 

For any $t \geq 0$ and $\Lambda \subset \mathbb{Z}^d$, we denote 
$\mathcal{F}_{t,\Lambda} = \sigma(t_{x,n},B_{x,n},x \in \Lambda,t_{x,n} \leq t)$ the 
$\sigma$-algebra of the exponential times and Bernoulli variables in the domain $\Lambda$ 
between time 0 and time $t$. We notice that if $\eta_0$ is deterministic, for any $x \in \mathbb{Z}^d$, 
$\eta_t(x)$ depends only on the $t_{x,n},B_{x,n}$ with $t_{x,n} \leq t$ and 
on the state of sites ``below'' $x$: $x-e_1,\dots,x-e_d$, which in turn depends only on the 
$\eta_0(x-e_i)$, $t_{x-e_i,n},B_{x-e_i,n}$ with $t_{x-e_i,n} \leq t$ and on the state of the sites ``below'' 
the $x-e_i$, etc. Therefore $\eta_t(x)$ depends only on $\eta_0$ and on the 
$t_{y,n},B_{y,n}$ with $t_{y,n} \leq t$ and $y \in x + (-\mathbb{N})^d$, 
hence $\eta_t(x)$ is $\mathcal{F}_{t,x + (-\mathbb{N})^d}$-measurable.

\medskip

We will call $\mu$ the product $\mathrm{Bernoulli}(p)$ measure on the configuration space $\{0,1\}^\Lambda$.
The expectation with respect to $\mu$ of a function 
$f : \{0,1\}^{\Lambda} \mapsto \mathbb{R}$, if it exists, will be denoted $\mu(f)$. 
$\mu$ is the equilibrium measure of the dynamics, which can be 
seen using reversibility, since the detailed balance is satisfied. 

We say that a measure $\nu$ on $\{0,1\}^{\mathbb{Z}^d}$ satisfies Condition $(\mathcal{C})$ when 
\[
  (\mathcal{C}) : \exists \, a,A > 0, \forall \, \ell \geq 0, 
  \nu(\forall \, x \in \{-\lfloor \ell\rfloor,\dots,0\}^d, \eta(x)=1) \leq A e^{-a\ell}.
\]

\begin{remark}
  The set of measures satisfying $(\mathcal{C})$ includes 
  \begin{itemize}
  \item the $\delta_\eta$ for any $\eta \in \{0,1\}^{\mathbb{Z}^d}$ such that there exists 
  $x=(x_1,\dots,x_d) \in (-\mathbb{N})^d$ with $\eta(x) = 0$. 
  This is the minimal condition on $\eta$ for which to expect convergence to equilibrium, 
  since if the initial configuration contains only ones, there can be no updates, 
  hence the dynamics is blocked. 
  \item the product $\mathrm{Bernoulli}(p')$ measures with $p' \in [0,1[$, 
  which are particularly relevant for physicists (see \cite{Leonard_et_al2007}). 
  \end{itemize}
\end{remark}

We can now state the main result of the paper, the convergence of the dynamics to equilibrium:

\begin{theorem}\label{thm_convergence}
  For any measure $\nu$ on $\{0,1\}^{\mathbb{Z}^d}$ satisfying 
  $(\mathcal{C})$, there exist constants $\chi = \chi(p) > 0$, 
  $c_1 = c_1(p,\nu) > 0$ and $C_1 = C_1(p,\nu) > 0$ such that, for any $t \geq 0$ 
  and any $f : \{0,1\}^{\mathbb{Z}^d} \mapsto \mathbb{R}$ 
  with $\mathrm{supp}(f) \subset \Lambda(\chi t^{1/d})$, 
  \[
  \int_{\{0,1\}^{\mathbb{Z}^d}}\left| \mathbb{E}_\eta(f(\eta_t)) - \mu(f) \right| 
  \mathrm{d}\nu(\eta) \leq C_1 \|f\|_\infty e^{-c_1 t}.
  \]
\end{theorem}

\begin{remark}
  With only minor modifications in the proof, one can also show exponential 
  convergence of the quantity $\int_{\{0,1\}^{\mathbb{Z}^d}}\left| \mathbb{E}_\eta(f(\eta_t)) - \mu(f) \right|^\gamma 
  \mathrm{d}\nu(\eta)$ for any $\gamma > 0$. 
\end{remark}

Another quantity of interest is the persistence function. If $\nu$ is the law of the 
initial configuration and $x \in \mathbb{Z}^d$, the corresponding persistence function can be defined as 
$F_{\nu,x}(t) = \mathbb{P}_\nu(\tau_x > t)$ for any $t \geq 0$, where $\tau_x$ is the 
first time there is an update at $x$. The persistence function is a
``measure of the mobility of the system'': the more the spin at $x$ can change, 
the faster it will decrease. Theorem \ref{thm_convergence} allows to prove exponential decay of the 
persistence function:

\begin{corollary}\label{cor_persistance}
  For any measure $\nu$ on $\{0,1\}^{\mathbb{Z}^d}$ satisfying 
  $(\mathcal{C})$, there exist constants $\chi = \chi(p) > 0$, $c_2 = c_2(p,\nu) > 0$ and 
  $C_2 = C_2(p,\nu) > 0$ such that for any $t \geq 0$ and any 
  $x\in\Lambda(\chi t^{1/d})$, $F_{\nu,x}(t) \leq C_2 e^{-c_2 t}$.
\end{corollary}

\begin{remark}
  The decay of the persistence function can not be faster than exponential, because 
  $\tau_x \geq t_{x,1}$, thus $F_{\nu,x}(t) \geq \mathbb{P}_\nu(t_{x,1} \geq t) = e^{-t}$. 
  Moreover, since the spin of a site $x$ will remain in its initial 
  state until $\tau_x$, the convergence to equilibrium can not be faster than exponential. 
  Consequently, the exponential speed is the actual speed.
\end{remark}

\begin{remark}
  In Theorem \ref{thm_convergence} and Corollary \ref{cor_persistance}, one could replace 
  $\Lambda(\chi t^{1/d})$ with any box of the form $(\prod_{i=1}^d\{0,\dots,a_i \}) \setminus \{(0,\dots,0)\}$, 
  $a_1,\dots,a_d \in \mathbb{N}$, $\prod_{i=1}^d(a_i+1)-1 \leq 2^d\chi^d t$.
\end{remark}

\section{Proof of Theorem \ref{thm_convergence}}\label{sec_preuve_thm}

The proof of the theorem can be divided in three steps. Firstly, we use a novel 
argument to find a site of $(-\mathbb{N})^d$ at distance $O(t)$ 
from the origin that remains at zero for a total time 
$\Omega(t)$ between time 0 and time $t$ (Section \ref{subsection_zero_initial}). 
Afterwards, we use sequentially a result of \cite{Chleboun_et_al2015} to prove that the origin also 
stays at zero for a time $\Omega(t)$ (Section \ref{subsection_origine_azero}). 
Finally, we end the proof of the theorem with the help of a formula derived in \cite{Chleboun_et_al2015}.

\subsection{Finding a site that stays at zero for a time $\Omega(t)$}\label{subsection_zero_initial}

For any $t \geq 0$ and $\alpha > 0$, 
we denote $D = D(t,\alpha) = \{-\lfloor 2d \alpha t \rfloor,\dots,0\}^d$. 
For any $x \in \mathbb{Z}^d$,  we denote $\mathcal{T}_t(x) = \int_0^t 
\mathbb{1}_{\{\eta_s(x)=0\}}\mathrm{d}s$ the time that $x$ spends at zero 
between time 0 and time $t$. We also define 
$\mathcal{G}=\{\exists x \in D \,|\, \mathcal{T}_t(x) \geq \frac{1-p}{4}t\}$. 
We then have 

\begin{proposition}\label{lemme_zero_initial}
  For any $\alpha > 0$, there exist constants 
  $c_3 = c_3(p,\alpha) > 0$ and $C_3 = C_3(p,\alpha) > 0$ such that for any $t \geq 0$, for any 
  $\eta \in \{0,1\}^{\mathbb{Z}^d}$ such that there exists 
  $x \in \{-\lfloor \alpha t\rfloor,\dots,0\}^d$ with 
  $\eta(x)=0$, $\mathbb{P}_\eta(\mathcal{G}^c) \leq C_3 e^{-c_3 t}$.
\end{proposition}

\begin{figure}
  \begin{center}
   \begin{tikzpicture}[scale=0.25]
    % Z^2
    \draw [ultra thin, color=gray!40] (-18,-18) grid (2,2) ;
    \draw[->,>=latex] (-18,0)--(2,0) ;
    \draw[->,>=latex] (0,-18)--(0,2) ;
    \draw (0,0) node [above right] {$0$} ; 
    \draw (-16,0) node{$\shortmid$} node [above] {$-\lfloor \beta  t \rfloor$} ;
    % boxes
    \draw[thick] (0.5,-4.5)--(-4.5,-4.5)--(-4.5,0.5)--(0.5,0.5)--cycle ;
    \draw (0.5,-4) node [right] {$\{-\lfloor \alpha t\rfloor,\dots,0\}^d$};
    \draw[thick] (0.5,-15.5)--(-15.5,-15.5)--(-15.5,0.5)--(0.5,0.5)--cycle ;
    \draw (0.5,-15) node [above right] {$D'$};
    \draw[thick] (0.5,-16.5)--(-16.5,-16.5)--(-16.5,0.5)--(0.5,0.5)--cycle ;
    \draw (0.5,-16) node [right] {$D$};
    % hyperplane 
    \draw (-12,0.5)--(0.5,-12)--(0,-12.5)--(-12.5,0)--cycle ;
    \draw (-6,-6) node [below left] {$H_k$} ;
    % x 
    \draw (-3,-2) node{$\ast$} node [right] {$x$} ;
    % oriented path
    \draw[->,>=latex,very thick] (-3,-2)--(-3,-3)--(-4,-3)--(-4,-6)--(-4,-10)--(-6,-10)--(-6,-11)--(-8,-11) ;
    \draw[->,>=latex,very thick] (-6,-11)--(-8,-11)--(-8,-12)--(-9,-12)--(-9,-14)--(-10,-14)--(-10,-16) ;
    \draw[->,>=latex,very thick] (-4,-5)--(-4,-6) ;
    \draw[->,>=latex,very thick] (-7,-11)--(-8,-11) ;
   \end{tikzpicture}
  \end{center}
 
  \caption{The setting of the proof of Proposition \ref{lemme_zero_initial} for $d=2$. 
  The thick squares represent $\{-\lfloor \alpha t\rfloor,\dots,0\}^d$, $D'$ and $D$ (from smallest to largest). 
  An oriented path joining $x \in \{-\lfloor \alpha t\rfloor,\dots,0\}^d$ to $D \setminus D'$ (thick arrows) 
  intersects any $H_k$ with $k \in \{d\lfloor \alpha t \rfloor,\dots,\lfloor \beta  t \rfloor\}$ 
  (sloped rectangle).}
  \label{figure}
\end{figure}
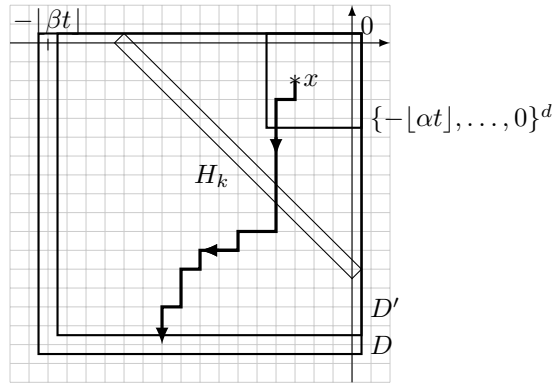

\begin{proof}
  The setting of the proof is illustrated in Figure \ref{figure}. 
  We set $\alpha > 0$. It is enough to prove the proposition for $t \geq 1/(2d\alpha-\alpha)$, 
  so we fix $t \geq 1/(2d\alpha-\alpha)$. Let $\eta \in \{0,1\}^{\mathbb{Z}^d}$ with 
  $x \in \{-\lfloor \alpha t\rfloor,\dots,0\}^d$ such that $\eta(x)=0$ be the initial configuration.
  We define $E = \{y \in D \,|$ there was an 
  update at $y$ in the time interval $[0,t/2]\}$.
  Moreover, an \emph{oriented path} will be a sequence 
  of sites $(x^{(1)},\dots,x^{(n)})$ with $n \in \mathbb{N}^*$ such that 
  for any $k \in \{1,\dots,n-1\}$, there exists $i\in\{1,\dots,d\}$ with $x^{(k+1)} = x^{(k)}-e_i$. 
  Furthermore, writing $\beta  = 2d\alpha$, we can define $D' = \{-\lfloor \beta  t \rfloor+1,\dots,0\}^d$. 
  Since $t \geq 1/(2d\alpha-\alpha)$, $2d\alpha t -1 \geq \alpha t$, 
  so $-\lfloor \beta  t \rfloor+1 \leq -\lfloor \alpha t\rfloor$, thus $x \in D'$.

  The proof of Proposition \ref{lemme_zero_initial} relies on the following 
  auxiliary lemma, whose proof will be postponed until 
  after the proof of Proposition \ref{lemme_zero_initial}:

  \begin{lemma}\label{lemme_chemins_orientes}
    If no site in $D$ stays at zero during the time interval 
    $[0,t/2]$, then there exists an oriented path in $E$ joining $x$ to $D \setminus D'$.
  \end{lemma}

  This auxiliary lemma implies that we either get a site satisfying $\mathcal{G}$, 
  or a path of $\Omega(t)$ sites that were updated before time $t/2$. In the latter case, 
  the orientation of the model allows us to use a conditioning which yields that the 
  probabilty that none of the sites of the path stays at zero for a time $\frac{1-p}{4}t$ 
  is the product of the probabilities for each of the sites not to stay at zero 
  for a time $\frac{1-p}{4}t$, and we can prove that this probabilty is strictly smaller than one. We now make this
  argument precise to prove Lemma \ref{lemme_zero_initial}.

  For any $k \in \{0,\dots,d\lfloor \beta  t \rfloor\}$, we define the ``diagonal hyperplane'' 
  $H_k = \{(x_1,\dots,x_d) \in D \,|\,x_1+\cdots+x_d = -k\}$ (see Figure \ref{figure}) and 
  we denote $\mathcal{U}_k = \{H_k \cap E \neq \emptyset\}$.
  If $\mathcal{G}^c$ occurs, no site of $D$ 
  can stay at zero during the whole time interval $[0,t/2]$, hence by Lemma 
  \ref{lemme_chemins_orientes} there exists an oriented path in $E$ joining $x$ to 
  $D \setminus D'$. Since $x\in\bigcup_{k=0}^{d\lfloor \alpha t \rfloor}H_k$ and 
  $D \setminus D' \subset \bigcup_{k=\lfloor \beta  t \rfloor}^{d\lfloor \beta  t \rfloor}H_k$, 
  $E$ intersects all the $H_k$ for $k \in \{d\lfloor \alpha t \rfloor,\dots,\lfloor \beta  t \rfloor\}$. 
  This implies $\mathcal{G}^c \subset \bigcap_{k = d\lfloor \alpha t \rfloor}^{\lfloor \beta  t \rfloor} \mathcal{U}_k$.
  Furthermore, for any $k \in \{0,\dots,d\lfloor \beta  t \rfloor\}$, we may define 
  $\mathcal{G}_{k} = \{\exists x \in H_k, \mathcal{T}_t(x) \geq \frac{1-p}{4}t\}$, then 
  $\mathcal{G}^c \subset \bigcap_{k=d\lfloor \alpha t \rfloor}^{\lfloor \beta  t \rfloor}\mathcal{G}_{k}^c$.
  We deduce $\mathcal{G}^c \subset \bigcap_{k=d\lfloor \alpha t \rfloor}
  ^{\lfloor \beta  t \rfloor} (\mathcal{U}_k \cap \mathcal{G}_{k}^c)$, so 
  \[
    \mathbb{P}_\eta (\mathcal{G}^c) \leq 
    \mathbb{E}_\eta \left(\prod_{k=d\lfloor \alpha t \rfloor}^{\lfloor \beta  t \rfloor} 
    (\mathbb{1}_{\mathcal{U}_k} \mathbb{1}_{\mathcal{G}_{k}^c})\right). 
  \]

  For all $k \in \{d\lfloor \alpha t \rfloor,\dots,\lfloor \beta  t \rfloor\}$, we define a $\sigma$-algebra 
  $\mathcal{F}_k = \sigma(\mathcal{F}_{t,\Lambda}, 
  \sigma(t_{x,n},x \in H_k,t_{x,n} \leq t/2))$, where 
  $\Lambda = \{(x_1,\dots,x_d) \in (-\mathbb{N})^d \,|\, x_1+\cdots+x_d < -k\}$.
  For any $\ell \in \{d\lfloor \alpha t \rfloor,\dots,\lfloor \beta  t \rfloor\}$ with $\ell > k$, 
  one can see that everything that happens at the sites in $H_{\ell}$ between 
  times 0 and $t$ is $\mathcal{F}_k$-measurable, 
  thus $\mathcal{U}_{\ell}$ and $\mathcal{G}_{\ell}^c$ are $\mathcal{F}_k$-measurable. 
  Moreover, for any $x \in H_k$, the spins of the $x-e_i$, $i \in \{1,\dots,d\}$ in the time interval 
  $[0,t/2]$ are $\mathcal{F}_k$-measurable and 
  the $t_{x,n} \leq t/2$ are also $\mathcal{F}_k$-measurable. Therefore the event 
  $\{$there was an update at $x$ between time 0 and time $t/2\}$ is 
  $\mathcal{F}_k$-measurable, hence $\mathcal{U}_k$ is $\mathcal{F}_k$-measurable. Consequently, 
  \[
    \mathbb{P}_\eta (\mathcal{G}^c) \leq 
    \mathbb{E}_\eta \left(\left.\mathbb{E}_\eta \left(\mathbb{1}_{\mathcal{G}_{d\lfloor \alpha t \rfloor}^c}
    \right|\mathcal{F}_{d\lfloor \alpha t \rfloor}\right)
    \mathbb{1}_{\mathcal{U}_{d\lfloor \alpha t \rfloor}}
    \prod_{k=d\lfloor \alpha t \rfloor+1}^{\lfloor \beta  t \rfloor} (\mathbb{1}_{\mathcal{U}_k} 
    \mathbb{1}_{\mathcal{G}_{k}^c}) \right).
  \]
  Therefore, if we can find a constant $c_3' = c_3'(p) > 0$ such that 
  \begin{equation}\label{eq_lzi_principale}
    \forall k \in \{d\lfloor \alpha t \rfloor,\dots,\lfloor \beta  t \rfloor\}, \mathbb{1}_{\mathcal{U}_k}
    \mathbb{E}_\eta \left(\mathbb{1}_{\mathcal{G}_{k}^c}|\mathcal{F}_{k}\right) \leq e^{-c_3'}
  \end{equation}
  then we have 
  \[
    \mathbb{P}_\eta (\mathcal{G}^c) \leq e^{-c_3'} 
    \mathbb{E}_\eta \left(\prod_{k=d\lfloor \alpha t \rfloor+1}^{\lfloor \beta  t \rfloor} 
    (\mathbb{1}_{\mathcal{U}_k} \mathbb{1}_{\mathcal{G}_{k}^c})\right),
  \]
  so by a simple induction $\mathbb{P}_\eta (\mathcal{G}^c) 
  \leq e^{-c_3'(\lfloor \beta  t \rfloor+1-d\lfloor \alpha t \rfloor)} 
  \leq e^{-c_3'(\beta  t-d\alpha t)} =  e^{-c_3'd\alpha t}$, which is Proposition \ref{lemme_zero_initial}.

  Consequently, we only need to prove (\ref{eq_lzi_principale}).
  Let $k \in \{d\lfloor \alpha t \rfloor,\dots,\lfloor \beta  t \rfloor\}$. For any $x \in H_k$, 
  if the state of the $x-e_i$, $i \in \{1,\dots,d\}$ between time 0 and time $t$ is known, 
  and if the $t_{x,n} \leq t/2$ are also known, the state of $x$ between time 
  0 and time $t$ depends only on the $t/2 < t_{x,n} \leq t$ and 
  on the $B_{x,n}$ such that $t_{x,n} \leq t$. Therefore, conditionally on 
  $\mathcal{F}_k$, the state of $x$ between time 0 and time $t$ depends only on 
  $\{t/2 < t_{x,n} \leq t\}\cup \{B_{x,n}\,|\,t_{x,n} \leq t\}$. Moreover, 
  these sets for $x \in H_k$ are mutually independent conditionally on $\mathcal{F}_k$, hence 
  the states of the $x \in H_k$ between time 0 and time $t$ are mutually 
  independent conditionally on $\mathcal{F}_k$, which implies 
  \begin{equation}\label{eq_lzi_sepsites}
   \mathbb{1}_{\mathcal{U}_k}\mathbb{E}_\eta \left(\mathbb{1}_{\mathcal{G}_{k}^c}|\mathcal{F}_{k}\right)
    = \mathbb{1}_{\mathcal{U}_k}\!\prod_{x \in H_k}\!\mathbb{P}_\eta\!\left(\left.\!
    \mathcal{T}_t(x) < \frac{1-p}{4}t\right|\mathcal{F}_{k}\!\right)
    \leq \mathbb{1}_{\mathcal{U}_k}\!\prod_{x \in H_k \cap E}\!\mathbb{P}_\eta\!\left(\left.\!
    \mathcal{T}_t(x) < \frac{1-p}{4}t\right|\mathcal{F}_{k}\!\right)\!.
  \end{equation}

  In addition, for $x \in H_k \cap E$, we have the following (in the second inequality we use the Markov inequality):  
  \[
    \mathbb{P}_\eta\left(\left. \mathcal{T}_t(x) < \frac{1-p}{4}t\right|\mathcal{F}_{k}\right) 
    \leq \mathbb{P}_\eta\left(\left.\int_{t/2}^{t}\mathbb{1}_{\{\eta_s(x)=0\}}\mathrm{d}s 
    < \frac{1-p}{4}t\right|\mathcal{F}_{k}\right)
  \]
  \[
    \leq \frac{\mathbb{E}_\eta\left(\left.\int_{t/2}^{t}
    \mathbb{1}_{\{\eta_s(x)=1\}}\mathrm{d}s \right|\mathcal{F}_{k}\right)}{\frac{t}{2}-\frac{1-p}{4}t}
     = \frac{\int_{t/2}^{t} \mathbb{P}_\eta (\eta_s(x)=1|\mathcal{F}_{k})\mathrm{d}s}
    {\left(1-\frac{1-p}{2}\right)\frac{t}{2}}.
  \]
  Furthermore, for $s \in [t/2,t]$, since $x \in H_k \cap E$, conditionally on $\mathcal{F}_k$ 
  we know that there was an update at $x$ before time $s$, but not the associated Bernoulli variable, hence 
  $\mathbb{P}_\eta (\eta_s(x)=1|\mathcal{F}_{k}) = p$. This implies 
  \[
    \mathbb{P}_\eta\left(\left. \mathcal{T}_t(x) < \frac{1-p}{4}t\right|\mathcal{F}_{k}\right) \leq 
    \frac{\int_{t/2}^{t} p \mathrm{d}s}{\left(1-\frac{1-p}{2}\right)\frac{t}{2}}
    = \frac{p}{1-\frac{1-p}{2}}.
  \]
  Moreover, $\frac{p}{1-\frac{1-p}{2}} = \frac{2p}{1+p} < 1$, 
  hence if we write $c_3' = -\ln(\frac{p}{1-\frac{1-p}{2}})$, we have $c_3' > 0$ and 
  $\mathbb{P}_\eta(\mathcal{T}_t(x) < \frac{1-p}{4}t|\mathcal{F}_{k}) \leq e^{-c_3'}$. 
  Consequently, (\ref{eq_lzi_sepsites}) yields 
  \[
    \mathbb{1}_{\mathcal{U}_k}\mathbb{E}_\eta \left(\mathbb{1}_{\mathcal{G}_{k}^c}|\mathcal{F}_{k}\right) \leq 
    \mathbb{1}_{\mathcal{U}_k}\prod_{x \in H_k \cap E} e^{-c_3'} = 
    \mathbb{1}_{\mathcal{U}_k}e^{-c_3'| H_k \cap E|}.
  \]
  Finally, $\mathcal{U}_k$ indicates that $H_k \cap E \neq 0$, thus 
  $\mathbb{1}_{\mathcal{U}_k}\mathbb{E}_\eta (\mathbb{1}_{\mathcal{G}_{k}^c}|\mathcal{F}_{k}) \leq 
  \mathbb{1}_{\mathcal{U}_k} e^{-c_3'} \leq e^{-c_3'}$ with $c_3' > 0$ depending only on $p$, 
  which is (\ref{eq_lzi_principale}).
\end{proof}

\begin{proof}[Proof of Lemma \ref{lemme_chemins_orientes}.]
  Let us suppose that no site of $D$ stays at zero during the time interval $[0,t/2]$.
  Then $E$ contains $x$, because $x \in D$ and if there was no update at $x$ between time 0 and 
  time $t/2$, the spin of $x$ would stay during this whole time interval at its initial state of 0, 
  which does not happen by assumption.
  We are going to show that if we have an oriented path in $E$ starting from $x$ that does not reach 
  $D \setminus D'$, we can add a site at its end in a way we still have an oriented path in $E$. 
  This is enough, because from the path composed only of $x$ 
  we can do at most $d\lfloor \beta  t \rfloor$ steps before reaching $D \setminus D'$.
  Thus we consider an oriented path in $E$ starting from $x$ that does not reach $D \setminus D'$. 
  Let us call $y$ its last site; we have $y \in D'$. 
  Since $y \in E$, $y$ was updated between time 0 and time 
  $t/2$. This implies that one of the $y-e_i$, $i \in \{1,\dots,d\}$, that we may call 
  $y'$, was at zero at the moment of the update.  Moreover, $y \in D'$, hence $y' \in D$.
  There are two possibilities: 
  \begin{itemize}
  \item Either the spin of $y'$ was not zero in the initial configuration. Then there was an 
  update at $y'$ before the update at $y$, hence before time $t/2$, so 
  since $y' \in D$, $y' \in E$.
  \item Or the spin at $y'$ was zero in the initial configuration. In this case, if there was no 
  update at $y'$ before time $t/2$, $y'$ stayed at 0 during the whole time interval 
  $[0,t/2]$. However $y' \in D$, 
  so this is impossible by assumption. Therefore there was an update 
  at $y'$ before time $t/2$, which implies $y' \in E$.
  \end{itemize}
  Therefore $y' \in E$ in both cases, which allows to add a site to the 
  path and ends the proof of Lemma \ref{lemme_chemins_orientes}.
\end{proof}

\subsection{Proving the origin stays at zero for a time $\Omega(t)$}\label{subsection_origine_azero}

In this section, we will use Proposition \ref{lemme_zero_initial} to prove the following result: 

\begin{lemma}\label{lemme_FK}
  There exist constants $\delta = \delta(p) \in ]0,1[$, $\alpha = \alpha(p) > 0$, $c_4 = c_4(p) > 0$ and 
  $C_4 = C_4(p) > 0$ such that for any $t \geq 0$, for any 
  $\eta \in \{0,1\}^{\mathbb{Z}^d}$ such that there exists 
  $x \in \{-\lfloor \alpha t\rfloor,\dots,0\}^d$ with 
  $\eta(x)=0$, $\mathbb{P}_\eta(\mathcal{T}_t(0) \leq \frac{1-p}{4}\delta^d t) \leq C_4 e^{-c_4 t}$.
\end{lemma}

\begin{proof}
  Let $t \geq 0$. Thanks to Proposition \ref{lemme_zero_initial}, 
  for any $\alpha > 0$ and $\eta \in \{0,1\}^{\mathbb{Z}^d}$ such that there exists 
  $x \in \{-\lfloor \alpha t\rfloor,\dots,0\}^d$ with $\eta(x)=0$, we have 
  $\mathbb{P}_\eta(\mathcal{G}^c) \leq C_3 e^{-c_3 t}$ 
  with $c_3, C_3 > 0$ depending only on $p$ and $\alpha$. Therefore, it is enough to find 
  $\delta = \delta(p) \in ]0,1[$, $\alpha = \alpha(p) > 0$, 
  $C_4' = C_4'(p) > 0$ and $c_4' = c_4'(p) > 0$ depending only on $p$ such 
  that for $\eta \in \{0,1\}^{\mathbb{Z}^d}$ we have 
  $\mathbb{P}_\eta(\mathcal{G},\mathcal{T}_t(0) \leq \frac{1-p}{4}\delta^d t) 
  \leq C_4' e^{-c_4' t}$.

  Moreover, for any $\delta \in ]0,1[$, $\alpha > 0$ and $\eta \in \{0,1\}^{\mathbb{Z}^d}$, we have 
  \begin{equation}\label{eq_lFK_sepsites}
    \mathbb{P}_\eta\left(\mathcal{G},\mathcal{T}_t(0) \leq \frac{1-p}{4}\delta^d t\right) 
    \leq \sum_{y \in D} 
    \mathbb{P}_\eta\left(\mathcal{T}_t(y)\geq\frac{1-p}{4}t,
    \mathcal{T}_t(0) \leq \frac{1-p}{4}\delta^d t\right).
  \end{equation}
  For $y = (y_1,\dots,y_d) \in D$, we define the following 
  sequence of sites: $y^{(0)}=y,y^{(1)} = (0,y_2,\dots,y_d), 
  y^{(2)} = (0,0,y_3,\dots,y_d),\dots,y^{(d)} = (0,\dots,0)$. We then have 
  \begin{equation}\label{eq_lFK_sepdim}
    \begin{split}
      \mathbb{P}_\eta\left(\mathcal{T}_t(y)\geq\frac{1-p}{4}t,\mathcal{T}_t(0) \leq \frac{1-p}{4}\delta^d t\right)\qquad\qquad \\
      \leq \sum_{i=1}^d \mathbb{P}_\eta\left(\mathcal{T}_t(y^{(i-1)})\geq\delta^{i-1}\frac{1-p}{4}t,
      \mathcal{T}_t(y^{(i)}) \leq \frac{1-p}{4}\delta^i t\right).
    \end{split}
  \end{equation}

  To deal with this expression, we are going to use Lemma 4.9 of \cite{Chleboun_et_al2015}.
  This lemma yields that there exist constants $\delta \in ]0,1[$ and $c > 0$ depending only on $p$ 
  such that for any $i \in \{1,\dots,d\}$, defining $I_i = \{(0,\dots,j,y_{i+1},\dots,y_d) | 
  j \in \{y_{i}+1,\dots,0\}\}$ if $y_i \neq 0$ and $I_i = \emptyset$ if $y_i = 0$, 
  \[
    \mathbb{P}_\eta (\mathcal{T}_t(y^{(i)}) \leq \delta \mathcal{T}_t(y^{(i-1)})|\mathcal{F}_{t,I_i^c}) 
    \leq \frac{1}{(p \wedge (1-p))^{|y_i|}} e^{-c\mathcal{T}_t(y^{(i-1)})}.
  \]
  (Actually, this lemma was proven for a dynamics in $\mathbb{N}^d$, but the proof works in 
  $\mathbb{Z}^d$ with only minor modifications.)

  Therefore we can set $\delta$ to the value given by \cite{Chleboun_et_al2015}, and obtain the following 
  (in the first inequality we use that $\mathcal{T}_t(y^{(i-1)})$ is 
  $\mathcal{F}_{t,y^{(i-1)}+(-\mathbb{N})^d}$-measurable, hence $\mathcal{F}_{t,I_i^c}$-measurable): 
  \[
    \mathbb{P}_\eta\left(\mathcal{T}_t(y^{(i-1)})\geq\delta^{i-1}\frac{1-p}{4}t,
    \mathcal{T}_t(y^{(i)}) \leq \frac{1-p}{4}\delta^i t\right) 
  \]
  \[
    \leq \mathbb{E}_\eta(\mathbb{1}_{\{\mathcal{T}_t(y^{(i-1)})\geq\delta^{i-1}\frac{1-p}{4}t\}}
    \mathbb{P}_\eta (\mathcal{T}_t(y^{(i)}) \leq \delta \mathcal{T}_t(y^{(i-1)})|\mathcal{F}_{t,I_i^c})) 
  \]
  \[
    \leq \mathbb{E}_\eta\left(\mathbb{1}_{\{\mathcal{T}_t(y^{(i-1)})\geq\delta^{i-1}\frac{1-p}{4}t\}}
    \frac{1}{(p \wedge (1-p))^{|y_i|}} e^{-c\mathcal{T}_t(y^{(i-1)})} \right) 
    \leq \frac{1}{(p \wedge (1-p))^{|y_i|}} e^{-c\delta^{i-1}\frac{1-p}{4}t}. 
  \]
  Moreover, since $y \in D$, $|y_i| \leq \lfloor 2d\alpha t \rfloor \leq 2d\alpha t$, so 
  if we set $\alpha = \frac{c(1-p)\delta^{d-1}}{-16d\ln(p\wedge(1-p))}$ 
  (which is positive and depends only on $p$), we obtain 
  $(p \wedge (1-p))^{|y_i|} \geq  e^{-\frac{c(1-p)\delta^{d-1}}{8}t}$, hence the last term in the display is bounded 
  by $e^{-\frac{c(1-p)\delta^{d-1}}{8}t}$. Therefore, by (\ref{eq_lFK_sepdim}), 
  \[
    \mathbb{P}_\eta\left(\mathcal{T}_t(y)\geq\frac{1-p}{4}t,
    \mathcal{T}_t(0) \leq \frac{1-p}{4}\delta^d t\right) 
    \leq d e^{-\frac{c(1-p)\delta^{d-1}}{8}t},
  \]
  so by (\ref{eq_lFK_sepsites}) 
  \[
    \mathbb{P}_\eta\left(\mathcal{G},\mathcal{T}_t(0) \leq \frac{1-p}{4}\delta^d t\right) 
    \leq |D| d e^{-\frac{c(1-p)\delta^{d-1}}{8}t} 
    = (\lfloor 2d\alpha t \rfloor+1)^d d e^{-\frac{c(1-p)\delta^{d-1}}{8}t}
  \]
  with $\frac{c(1-p)\delta^{d-1}}{8} > 0$ depending only on $p$ and $\alpha$ 
  depending only on $p$, so we get a suitable bound on 
  $\mathbb{P}_\eta(\mathcal{G},\mathcal{T}_t(0) \leq \frac{1-p}{4}\delta^d t)$.
\end{proof}

\subsection{Ending the proof of Theorem \ref{thm_convergence}}\label{subsection_preuve_thm}

Let $\nu$ a measure on $\{0,1\}^{\mathbb{Z}^d}$ satisfying $(\mathcal{C})$, 
$t \geq 0$ and $f : \{0,1\}^{\mathbb{Z}^d} \mapsto \mathbb{R}$ non constant 
with $\|f\|_\infty < \infty$. We denote $\mathcal{N}(\eta) = \{\exists x \in 
\{-\lfloor \alpha t \rfloor,\dots,0\}^d,\eta(x)=0\}$, where $\alpha = \alpha(p) > 0$ is 
given by Lemma \ref{lemme_FK}. We also denote $g = \frac{f - \mu(f)}{\|f - \mu(f)\|_\infty}$. Then 
\[
  \int_{\{0,1\}^{\mathbb{Z}^d}}\left| \mathbb{E}_\eta(f(\eta_t)) - \mu(f) \right| 
  \mathrm{d}\nu(\eta) = \|f - \mu(f)\|_\infty \int_{\{0,1\}^{\mathbb{Z}^d}}\left| \mathbb{E}_\eta(g(\eta_t)) \right| 
  \mathrm{d}\nu(\eta)
\]
\[
  \leq 2 \|f\|_\infty \left( \int_{\{0,1\}^{\mathbb{Z}^d}}|\mathbb{E}_\eta(g(\eta_t))|
  \mathbb{1}_{\mathcal{N}(\eta)^c}\mathrm{d}\nu(\eta) 
  + \int_{\{0,1\}^{\mathbb{Z}^d}}| \mathbb{E}_\eta(g(\eta_t))|
  \mathbb{1}_{\mathcal{N}(\eta)}\mathrm{d}\nu(\eta) \right).
\]
Moreover, since $\|\mu(g)\|_\infty = 1$ and  $\nu$ satisfies $(\mathcal{C})$, we can see that we have 
the following: $\int_{\{0,1\}^{\mathbb{Z}^d}}|\mathbb{E}_\eta(g(\eta_t))| 
\mathbb{1}_{\mathcal{N}(\eta)^c}\mathrm{d}\nu(\eta) \leq \nu(\mathcal{N}(\eta)^c) \leq A e^{-a \alpha t}$ 
with $A,a > 0$ depending only on $\nu$. 
  
Therefore, to prove Theorem \ref{thm_convergence}, it is enough to find $\chi > 0$ depending only on $p$ 
such that for any $f : \{0,1\}^{\mathbb{Z}^d} \mapsto \mathbb{R}$ non constant 
(if $f$ is constant the theorem is trivially true) with support in $\Lambda(\chi t^{1/d})$ 
(which automatically gives $\|f\|_\infty < \infty$) and any $\eta \in \{0,1\}^{\mathbb{Z}^d}$ 
such that $\mathcal{N}(\eta)$, $| \mathbb{E}_\eta(g(\eta_t))| \leq C_1' e^{-c_1' t}$ 
with $C_1', c_1' >0$ depending only on $p$. For $\chi > 0$, we set such $f$ and $\eta$. 
Since $\|g\|_\infty = 1$, for $\delta$ as in Lemma \ref{lemme_FK} we have 
\[
  \left| \mathbb{E}_\eta(g(\eta_t)) \right| 
  \leq \mathbb{P}_\eta\left(\mathcal{T}_t(0) \leq \frac{1-p}{4}\delta^d t\right) 
  + \left|\mathbb{E}_\eta\left(\mathbb{1}_{\{\mathcal{T}_t(0) > \frac{1-p}{4}\delta^d t\}}
  g(\eta_t)\right)\right|.
\]
In addition, since there is $x \in \{-\lfloor \alpha t\rfloor,\dots,0\}^d$ such that $\eta(x)=0$, by Lemma 
\ref{lemme_FK} we have $\mathbb{P}_\eta(\mathcal{T}_t(0) \leq \frac{1-p}{4}\delta^d t) 
\leq C_4e^{-c_4 t}$ with $C_4,c_4 > 0$ depending only on $p$. Consequently, it is enough to bound 
$|\mathbb{E}_\eta(\mathbb{1}_{\{\mathcal{T}_t(0) > \frac{1-p}{4}\delta^d t\}} g(\eta_t))|$.

Writing $\Lambda = \Lambda(\chi t^{1/d})$ for short, 
we notice that the event $\{\mathcal{T}_t(0) > \frac{1-p}{4}\delta^d t\}$ is 
$\mathcal{F}_{t,(-\mathbb{N})^d}$-measurable hence $\mathcal{F}_{t,\Lambda^c}$-measurable, 
which implies 
\[
  \left|\mathbb{E}_\eta\left(\mathbb{1}_{\{\mathcal{T}_t(0) > \frac{1-p}{4}\delta^d t\}}
  g(\eta_t)\right)\right| 
  = \left|\mathbb{E}_\eta\left(\mathbb{1}_{\{\mathcal{T}_t(0) > \frac{1-p}{4}\delta^d t\}}
  \mathbb{E}_\eta\left(g(\eta_t)|\mathcal{F}_{t,\Lambda^c}\right)\right)\right|,
\]
therefore 
\[
  \left|\mathbb{E}_\eta\left(\mathbb{1}_{\{\mathcal{T}_t(0) > \frac{1-p}{4}\delta^d t\}}
  g(\eta_t)\right)\right| 
\]
\[
  \leq \frac{1}{\min_{\sigma \in \{0,1\}^\Lambda}\mu(\sigma)} 
  \mathbb{E}_\eta \left(\mathbb{1}_{\{\mathcal{T}_t(0) > \frac{1-p}{4}\delta^d t\}}
  \sum_{\sigma \in \{0,1\}^\Lambda}\mu(\sigma)
  \mathbb{E}_{\sigma\cdot\eta}\left(g(\eta_t)|\mathcal{F}_{t,\Lambda^c}\right)\right),
\]
where $\sigma\cdot\eta$ is the configuration equal to $\sigma$ in $\Lambda$ and 
to $\eta$ in $\Lambda^c$. Furthermore, the reasoning of Equation (4.2) 
of \cite{Chleboun_et_al2015} and of the paragraphs around it yields that 
\[
  \sum_{\sigma \in \{0,1\}^\Lambda}\mu(\sigma)
  \mathbb{E}_{\sigma\cdot\eta}\left(g(\eta_t)|\mathcal{F}_{t,\Lambda^c}\right) 
  \leq e^{-\lambda \mathcal{T}_t(0)}
\]
where $\lambda$ is the \emph{spectral gap} of the East dynamics in $\Lambda$ 
where the spin of the origin is fixed at 0 and the other spins outside $\Lambda$ 
are at 1 (see Chapter 2 of \cite{Guionnet_et_al2002} for the definition of the spectral gap 
and Section 2.4 of \cite{Cancrini_et_al2008} for an introduction to the spectral gap in the particular 
context of kinetically constrained models). 
Moreover, one can use the argument of Section 6.2.2 of \cite{Chleboun_et_al_2014} 
on our $\Lambda$ instead of on a cube to obtain that 
$\lambda$ is bigger than the spectral gap $\lambda'$ of the one-dimensional East dynamics in 
$\{1,\dots,d\lfloor\chi t^{1/d}\rfloor\}$ with the origin fixed at zero.
To do that, one can use a forest instead of a tree and apply the fact 
that the spectral gap of a product dynamics is the 
minimum of the spectral gaps of the component dynamics 
(Theorem 2.5 of \cite{Guionnet_et_al2002}). Furthermore, Equation (3.3) of 
\cite{Chleboun_et_al_2014} yields that $\lambda'$ is bigger than the spectral gap $\lambda''$ of the East
dynamics in $\mathbb{Z}$, which depends only on $p$ and is positive by Theorem 6.1 of 
\cite{Cancrini_et_al2008}. 

Consequently, we have 
\[
  \left|\mathbb{E}_\eta\left(\mathbb{1}_{\{\mathcal{T}_t(0) > \frac{1-p}{4}\delta^d t\}}g(\eta_t)\right)\right| 
  \leq \frac{1}{\min_{\sigma \in \{0,1\}^\Lambda}\mu(\sigma)} 
  \mathbb{E}_\eta \left(\mathbb{1}_{\{\mathcal{T}_t(0) > \frac{1-p}{4}\delta^d t\}}
  e^{-\lambda'' \mathcal{T}_t(0)}\right) 
\]
\[
  \leq \frac{1}{(p \wedge (1-p))^{|\Lambda|}}e^{-\lambda'' \frac{1-p}{4}\delta^d t}.
\]
Moreover, $|\Lambda| \leq (\chi t^{1/d}+1)^d$ and we can suppose $\chi t^{1/d} \geq 1$, since if 
$\chi t^{1/d} < 1$, $|\Lambda|$ is empty and there is no non constant function with support in $\Lambda$. 
Therefore we get $|\Lambda| \leq (2\chi t^{1/d})^d = 2^d \chi^d t$.
Now, if we set $\chi = \frac{1}{2}(\frac{\lambda''(1-p)\delta^d}{-8\ln(p\wedge(1-p))})^{1/d}$, $\chi$ is positive 
and depends only on $p$, and we have $(p \wedge (1-p))^{|\Lambda|} \geq e^{-\frac{\lambda''(1-p)\delta^d}{8}t}$, 
thus 
\[
  \left|\mathbb{E}_\eta\left(\mathbb{1}_{\{\mathcal{T}_t(0) > \frac{1-p}{4}\delta^d t\}}
  g(\eta_t)\right)\right| 
  \leq e^{-\frac{\lambda''(1-p)\delta^d}{8}t} 
\]
with $\frac{\lambda''(1-p)\delta^d}{8}$ positive depending only on $p$, which ends the proof of 
Theorem \ref{thm_convergence}.

\section{Proof of Corollary \ref{cor_persistance}}\label{sec_preuve_cor}

This proof is inspired from the proof of Lemma A.3 of \cite{Chleboun_et_al_2014}.

Let $\nu$ a measure on $\{0,1\}^{\mathbb{Z}^d}$ satisfying $(\mathcal{C})$, 
$\chi$ as in Theorem \ref{thm_convergence}, $t \geq 0$, $x \in \Lambda(\chi t^{1/d})$. 
For any $\eta \in \{0,1\}^{\mathbb{Z}^d}$, we have 
\[
  \mathbb{E}_\eta(\eta_t(x))=\mathbb{E}_\eta(\eta_t(x)|\tau_x \leq t)\mathbb{P}_\eta(\tau_x \leq t) 
  + \mathbb{E}_\eta(\eta_t(x)|\tau_x > t)\mathbb{P}_\eta(\tau_x > t) 
\]
\[
  =p\mathbb{P}_\eta(\tau_x \leq t) + \eta(x)\mathbb{P}_\eta(\tau_x > t) 
  = p- p\mathbb{P}_\eta(\tau_x > t)+ \eta(x)\mathbb{P}_\eta(\tau_x > t) 
\]
since if $\tau_x \leq t$, $\eta_t(x)$ is a Bernoulli random variable of parameter $p$. Therefore, 
\[
  |\mathbb{E}_\eta(\eta_t(x))-p| = |\eta(x)-p|\mathbb{P}_\eta(\tau_x > t) 
  \geq (p \wedge (1-p))\mathbb{P}_\eta(\tau_x > t),
\]
and we deduce 
\[
  F_{\nu,x}(t) = \mathbb{P}_\nu(\tau_x > t) 
  = \int_{\{0,1\}^{\mathbb{Z}^d}}\mathbb{P}_\eta(\tau_x > t)\mathrm{d}\nu(\eta) 
\]
\[
  \leq \frac{1}{p \wedge (1-p)}\int_{\{0,1\}^{\mathbb{Z}^d}}|\mathbb{E}_\eta(\eta_t(x))-p|\mathrm{d}\nu(\eta)
  \leq \frac{1}{p \wedge (1-p)} C_1 e^{-c_1 t}
\]
by Theorem \ref{thm_convergence} with $C_1 > 0$ and $c_1 > 0$ depending only on $p$ and $\nu$.

%%%%%%%%%%%%%%%%%%%%%%%%%%%%%%%%%%%%%%%%%%%%%%%%%%%%%%%%%%%%%%%%%%%
%%                                                               %%
%% Use the two commands below for producing your bibliography    %%
%% with bibtex, then comment again the commands and include the  %%
%% content of the .bbl file in this file below the commands.     %%
%%                                                               %%
%%%%%%%%%%%%%%%%%%%%%%%%%%%%%%%%%%%%%%%%%%%%%%%%%%%%%%%%%%%%%%%%%%%

%\bibliographystyle{amsplain}
%\bibliography{yourbibfilename}

\begin{thebibliography}{10}

\bibitem{Berthier_et_al2005}
Ludovic Berthier and Juan~P. Garrahan, \emph{Numerical study of a fragile
  three-dimensional kinetically constrained model}, The journal of physical
  chemistry B \textbf{109} (2005), no.~8, 3578--3585.

\bibitem{Blondel2013}
Oriane Blondel, \emph{Front progression in the {E}ast model}, Stochastic
  processes and their applications \textbf{123} (2013), no.~9, 3430--3465.
  \MR{3071385}

\bibitem{Blondel_et_al2013}
Oriane Blondel, Nicoletta Cancrini, Fabio Martinelli, Cyril Roberto, and
  Cristina Toninelli, \emph{Fredrickson-{A}ndersen one spin facilitated model
  out of equilibrium}, Markov processes and related fields \textbf{19} (2013),
  no.~3, 383--406. \MR{3156958}

\bibitem{Blondel_et_al2018}
Oriane Blondel, Aurelia Deshayes, and Cristina Toninelli, \emph{Front evolution
  of the {F}redrickson-{A}ndersen one spin facilitated model}, Electronic
  journal of probability \textbf{24} (2019), 32. \MR{3903501}

\bibitem{Cancrini_et_al2008}
Nicoletta Cancrini, Fabio Martinelli, Cyril Roberto, and Cristina Toninelli,
  \emph{Kinetically constrained spin models}, Probability theory and related
  fields \textbf{140} (2008), no.~3--4, 459--504. \MR{2365481}

\bibitem{Cancrini_et_al2010}
Nicoletta Cancrini, Fabio Martinelli, Roberto~H. Schonmann, and Cristina
  Toninelli, \emph{Facilitated oriented spin models: some non equilibrium
  results}, Journal of statistical physics \textbf{138} (2010), no.~6,
  1109--1123. \MR{2601425}

\bibitem{Chleboun_et_al2015}
Paul Chleboun, Alessandra Faggionato, and Fabio Martinelli, \emph{Mixing time
  and local exponential ergodicity of the {E}ast-like process in
  $\mathds{Z}^d$}, Annales de la faculté des sciences de {T}oulouse
  \textbf{24} (2015), no.~4, 717--743. \MR{3434253}

\bibitem{Chleboun_et_al_2014}
Paul Chleboun, Alessandra Faggionato, and Fabio Martinelli, \emph{Relaxation to equilibrium of generalized {E}ast processes on
  $\mathds{Z}^d$: renormalization group analysis and energy-entropy
  competition}, Annals of probability \textbf{44} (2016), no.~3, 1817--1863.
  \MR{3502595}

\bibitem{review_East}
Alessandra Faggionato, Fabio Martinelli, Cyril Roberto, and Cristina Toninelli,
  \emph{The {E}ast model: recent results and new progresses}, Markov processes
  and related fields \textbf{19} (2013), no.~3, 407--452. \MR{3156959}

\bibitem{Ganguly_et_al2015}
Shirshendu Ganguly, Eyal Lubetzky, and Fabio Martinelli, \emph{Cutoff for the
  {E}ast process}, Communications in mathematical physics \textbf{335} (2015),
  no.~3, 1287--1322. \MR{3320314}

\bibitem{Garrahan_et_al}
Juan~P. Garrahan, Peter Sollich, and Cristina Toninelli, \emph{Kinetically
  constrained models}, Dynamical heterogeneities in glasses, colloids, and
  granular media, Oxford university press, 2011.

\bibitem{Guionnet_et_al2002}
Alice Guionnet and Boguslaw Zegarlinski, \emph{Lectures on logarithmic
  {S}obolev inequalities}, Séminaire de probabilités XXXVI, Springer, 2002,
  pp.~1--134. \MR{1971582}

\bibitem{Jackle_et_al1991}
Josef Jäckle and Siegfried Eisinger, \emph{A hierarchically constrained
  kinetic {I}sing model}, Zeitschrift für physik B condensed matter
  \textbf{84} (1991), no.~1, 115--124.

\bibitem{Leonard_et_al2007}
Sébastien Léonard, Peter Mayer, Peter Sollich, Ludovic Berthier, and Juan~P.
  Garrahan, \emph{Non-equilibrium dynamics of spin facilitated glass models},
  Journal of statistical mechanics: theory and experiment (2007), 07017.
  \MR{2335696}

\bibitem{Mountford_FA1f}
Thomas Mountford and Glauco Valle, \emph{Exponential convergence for the
  {F}redrickson-{A}ndersen one spin facilitated model}, Journal of theoretical
  probability \textbf{32} (2019), no.~1, 282--302. \MR{3908915}

\bibitem{Ritort_et_al}
Felix Ritort and Peter Sollich, \emph{Glassy dynamics of kinetically
  constrained models}, Advances in physics \textbf{52} (2003), no.~4, 219--342.

\bibitem{Swart2017}
Jan~M. Swart, \emph{A course in interacting particle systems},
  arXiv:1703.10007v1 (2017).
 
\end{thebibliography}

% add below the content of your .bbl file produced by bibtex.

%%%%%%%%%%%%%%%%%%%%%%%%%%%%%%%%%%%%%%%%%%%%%%%%%%%%%%%%%%%%%%%%%%%
%%                                                               %%
%% You may add acknowledgments (optional).                       %%
%%                                                               %%
%%%%%%%%%%%%%%%%%%%%%%%%%%%%%%%%%%%%%%%%%%%%%%%%%%%%%%%%%%%%%%%%%%%

\ACKNO{I would like to thank my PhD advisor, Cristina Toninelli.}

%%%%%%%%%%%%%%%%%%%%%%%%%%%%%%%%%%%%%%%%%%%%%%%%%%%%%%%%%%%%%%%%%%%
%%                                                               %%
%% You have reached the end of your document.                    %%
%%                                                               %%
%%%%%%%%%%%%%%%%%%%%%%%%%%%%%%%%%%%%%%%%%%%%%%%%%%%%%%%%%%%%%%%%%%%

\end{document}